\documentclass[a4paper,12pt]{amsart}
\usepackage{amsthm, amsfonts, amssymb, mathtools, color}
\usepackage[all]{xy}
\usepackage{fullpage}

\def\Z{\mathbb{Z}}
\def\Q{\mathbb{Q}}
\def\R{\mathbb{R}}

\def\P{\mathbb{P}}

\def\M{\mathcal{M}}
\def\OO{\mathcal{O}}

\def\til#1{\widetilde{#1}}
\def\ovl#1{\overline{#1}}
\def\gp{\mathrm{gp}}
\def\sat{\mathrm{sat}}

\DeclareMathOperator{\Spec}{Spec}

\DeclareMathOperator{\Proj}{Proj}

\DeclareMathOperator{\Hom}{Hom}

\DeclareMathOperator{\LSp}{S}
\DeclareMathOperator{\BSp}{Bl}
\def\dl{\langle\kern-.22em\langle}
\def\dr{\rangle\kern-.22em\rangle}
\def\dal{\langle\kern-.22em\langle}
\def\dar{\rangle\kern-.22em\rangle}
\def\dbl{[\kern-.12em[}
\def\dbr{]\kern-.12em]}
\def\dpl{(\kern-.2em(}
\def\dpr{)\kern-.2em)}

\def\dast{\ast\kern-.03em\ast}
\def\tast{\ast\kern-.26em\ast\kern-.26em\ast}
\title{Integral morphisms and log blow-ups}
\author{Fumiharu Kato}
\begin{document}
\theoremstyle{plain}
\newtheorem{thm}[subsection]{Theorem}
\newtheorem{prop}[subsection]{Proposition}
\newtheorem{lem}[subsection]{Lemma}
\newtheorem{cor}[subsection]{Corollary}
\newtheorem{probs}[subsection]{Problems}
\newtheorem{cla}{Claim}[subsection]
\theoremstyle{definition}
\newtheorem{dfn}[subsection]{Definition}
\newtheorem{ntn}[subsection]{Notation}
\newtheorem{exa}[subsection]{Example}
\newtheorem{exas}[subsection]{Examples}
\newtheorem{assum}[subsection]{Assumption}
\newtheorem{sit}[subsection]{Situation}
\theoremstyle{remark}
\newtheorem{rem}[subsection]{Remark}
\newtheorem{note}[subsection]{Note}

\begin{abstract}
This paper is a revision of the author's old preprint ``Exactness, integrality, and log modifications''. We will prove that any quasi-compact morphism of fs log schemes can be modified locally on the base to an integral morphism by base change by fs log blow-ups.
\end{abstract}

\maketitle

\setcounter{section}{0}
\section{Introduction}\label{sec-intro}
The aim of this paper is to prove the following theorem:
\begin{thm}\label{thm-main}
Let $f\colon X\rightarrow Y$ be a quasi-compact morphism of fs $(=$ fine and saturated$)$ log schemes.
Then for any $y\in Y$ there exists an \'etale neighborhood $V\rightarrow Y$ of $\ovl{y}$ $(=$ a separable closure of $y)$ and an fs log blow-up $V':=\BSp_{\mathcal{K}}(V)\rightarrow V$ along a coherent log ideal $\mathcal{K}\subset\M_V$, by which the fs base change $f_{V'}\colon X_{V'}\rightarrow V'$ is integral.
\end{thm}

Here, by {\em fs log blow-up} (resp.\ {\em fs base change}) we mean a log blow-up (resp.\ base change) in the category of fs log schemes; cf.\ Remark {\rm \ref{rem-normalization}}.

This theorem has been announced and proved in somewhat incomplete and inaccurate form in the author's old preprint \cite{FK}, a first draft of which has actually been written in 1997, and afterwards put in the arXiv in 1999.
Since then, mainly because the author has been away from log geometry, the paper has been kept unpublished; the author apologizes for all inconvenience caused thereby.
While there have been much progress and many new results in log geometry last two decades, the paper has sometimes been referred to.
Moreover, it seems, to the best of the author's knowledge, that the theorem itself has not yet been written anywhere, even in the foundational book \cite{AG2} by Ogus, and became folklore among experts.

In fact, the theorem is nowadays a consequence of combination of known results.
For example, Luc Illusie, Kazuya Kato, and Chikara Nakayama proved in \cite{IKN} (see also \cite[III.2.6.7]{AG2}) a weak version of the theorem, where ``integral'' is replaced by ``$\Q$-integral''.
Then by a further fs log blow-up of the base to make the log structure free (i.e., to make each stalk of $\ovl{\M}=\M/\OO^{\times}$ a free monoid), the resulting map becomes integral (cf.\ \cite[I.4.7.5]{AG2}).
Since, due to Nizio\l\ \cite[4.11]{Niz}, the composition of fs log blow-ups is again an fs log blow-up, this actually suffices to prove the theorem.

In the mean time, in August 2020, Michael Temkin asked the author some questions on the preprint, and suggested the final form of the theorem presented as above.
Based upon the fact that the theorem has to be referred to in a recent work \cite{ATW} of him and his coauthors, and that the theorem has not yet been presented in published form, he encouraged the author to revise the old preprint for publishing.
This is the situation from whence the present paper comes out, where we keep the original proof based on the technique of {\em toric flattening}, the original idea of which is attributed to Takeshi Kajiwara.

Let us mention some consequences of Theorem \ref{thm-main}.
Tsuji \cite[II.3.4]{TT} proved that any integral and quasi-compact morphism between fs log schemes can be made {\em saturated} by fs base change by ``multiplication-by-$n$'' map. 
Combined with our result, this yields the following:
\begin{cor}\label{cor-main1}
Let $f\colon X\rightarrow Y$ be a quasi-compact morphism between fs log schemes.
Then for any $y\in Y$ there exists a diagram
$$
\xymatrix{V''\ar[r]\ar[d]&V'\ar@{^{(}->}[r]\ar[d]&\BSp_K(V)\ar[d]\ar[r]&V\ar[d]\ar[r]&Y\\ \LSp(Q')\ar[r]_{\mu}&\LSp(Q')\ar@{^{(}->}[r]&\BSp_K(Q)\ar[r]_{\pi}&\LSp(Q)},
$$
with all vertical arrows strict and all squares cartesian in the category of fs log schemes, such that 
\begin{itemize}
\item $V\rightarrow Y$ is an \'etale neighborhood of $\ovl{y}$;
\item $\pi$ is an fs log blow-up of $\LSp(Q)=\Spec\Z[Q]$ $($with the log structure by $Q\rightarrow\Z[Q]$$)$ along an ideal $K\subset Q$;
\item $\LSp(Q')\hookrightarrow\BSp_K(Q)$ is an arbitrary affine patch of $\BSp_K(Q)$;
\item $\mu$ is the morphism induced from the multiplication-by-$n$ homomorphism $Q'\rightarrow Q'$ for some positive integrer $n$,
\end{itemize}
and that the fs base change $f_{V''}\colon X\times_YV''\rightarrow V''$ is saturated.
\end{cor}

Our second application is to log flatness.
\begin{cor}\label{cor-main2}
Let $f\colon X\rightarrow Y$ be a log flat and finitely presented morphism between fs log schemes.
Then for any $y\in Y$ there exists an \'etale neoghborhood $V\rightarrow Y$ of $\ovl{y}$ and an fs log blow-up $V'\rightarrow V$ such that the underlying morphism of the fs base change $f_{V'}\colon X\times_YV'\rightarrow V'$ is flat.
\end{cor}

Indeed, we may assume that $f$ is integral, log flat, and of finite presentation.
We may further assume that $f$ has a global chart by $h\colon Q\rightarrow P$, which is ``neat'' at a point $x\in X$ (as in Lemma \ref{lem-goodcharts} below) such that $Q\cong\ovl{\M}_{Y,\ovl{y}}$ ($y=f(x)$) and $\ovl{P}\cong\ovl{\M}_{X,\ovl{x}}$.
Then, since $\ovl{h}\colon \ovl{Q}=Q\rightarrow\ovl{P}$ is integral, so is $h$ (cf.\ Lemma \ref{lem-im} (5)).
Since $Q$ is sharp (i.e., $Q^{\times}=\{1\}$) and $h$ is local (i.e., $h^{-1}(P^{\times})=Q^{\times}$), the ring homomorphism $\Z[P]\rightarrow\Z[Q]$ is flat (cf.\ Lemma \ref{lem-im} (6)).
Since the log flatness implies that $X$ is flat over $Y\otimes_{\Z[Q]}\Z[P]$ (\cite[4.15]{Ols}), we deduce that $f$ itself is flat.

As for the interaction between log flatness and usual flatness, much has been studied recently by some authors. 
Among them, we refer to a preprint by Gillam \cite{Gil}.
It seems that, combining our result with many of the results therein, we can deduce several useful consequences.

Finally, let us remark that, if $Y$ in Theorem \ref{thm-main} is log regular, then $V$ can be equal to $Y$ itself, i.e., one can find an fs log blow-up of $Y$ itself that makes the morphism $f$ integral by fs base change.
This version of the theorem has been proven independently in \cite[3.6.11]{ATW}, which we will include, with a few comments, at the end of this paper (Proposition \ref{prop-ATW}).

\begin{rem}\label{rem-main}
The original paper \cite{FK} included an ``exactness version'' of the theorem, where ``integral'' is replaced by ``exact'', which we do not include in the present paper, since it follows immediately from the result in \cite{IKN} mentioned above.
\end{rem}

The composition of this paper is as follows.
In the next section, we will collect some basics of integral morphisms and neat charts.
In Section 3, we will overview log blow-ups.
We will then prove the main theorem in Section 4.

The original version of the paper owes much to Richard Pink, Takeshi Kajiwara, and Max Planck Institute f\"ur Mathematik in Bonn, Germany.
In addition, the preparation of the present version owes much to Michael Temkin for his encouragement, and to Chikara Nakayama for valuable discussions.

\subsection{Notation and conventions}\label{subsub-notation}
All rings and monoids are assumed to be commutative.
For a monoid $M$, we denote by $M^{\times}$ and $M^{\gp}$ the subgroup of invertible elements and the associated group, respectively, and write $\ovl{M}=M/M^{\times}$.

All sheaves on schemes are considered with respect to the \'etale topology. 
For a point $x$ of a scheme, we denote by $\ovl{x}$ a separable closure of $x$.
For a log scheme $X$, we denote by $\alpha_X\colon\M_X\rightarrow\OO_X$ the log structure of $X$, and write $\ovl{\M}_X=\M_X/\OO^{\times}_X$.
We denote by $\underline{X}$ the underlying scheme of $X$, which is also considered as a log scheme with the trivial log structure.

For a monoid $P$, we denote by $\LSp(P)$ the log scheme whose underlying scheme is the affine scheme $\Spec \Z[P]$ with the log structure induced from $P\rightarrow \Z[P]$.
($\LSp(P)$ is denoted by $\mathsf{A}_P$ in \cite{AG2}.)

\section{Integral homomorphisms}\label{sec-emim}
Let us first recall the definition of integral homomorphisms.

\begin{dfn}\label{dfn-emdm}
A homomorphism $h\colon Q\rightarrow P$ of integral monoids is said to be {\em integral} if, for any integral monoid $Q'$ and any homomorphism $Q\rightarrow Q'$, the push-out $P\oplus_QQ'$ in the category of monoids is an integral monoid.
\end{dfn}

It can be shown (\cite[(4.1)]{KK1}\cite[I.4.6.2]{AG2}) that $h\colon Q\rightarrow P$ is integral if and only if it has the following property: if $h(a_1)b_1=h(a_2)b_2$ for $a_1,a_2\in Q$ and $b_1 ,b_2\in P$, there exists $a_3,a_4\in Q$ and $b\in P$ such that $b_1=h(a_3)b$, $b_2=h(a_4)b$ and $a_1a_3=a_2a_4$.

\begin{lem}\label{lem-im}
$(1)$ The composition of integral homomorphisms is integral.
For a diagram $Q\stackrel{h}{\rightarrow}P\stackrel{k}{\rightarrow}R$ of integral monoids, if $k\circ h$ is integral and $k$ is exact, then $h$ is integral; if $ k\circ h$ is integral and $h$ is surjective, then $k$ is integral. 

$(2)$ The pushout of an integral homomorphism in the category of monoids is integral. 

$(3)$ If $P$ is an integral monoid, and $N\subset P$ is a submonoid, then the canonical map $P\rightarrow P/N$ is integral.

$(4)$ An integral and local homomorphism of integral monoids is exact.

$(5)$ A homomorphism $h\colon Q\rightarrow P$ of integral monoids is integral if and only if $\ovl{h}\colon\ovl{Q}\rightarrow\ovl{P}$ is integral.

$(6)$ A homomorphism $h\colon Q\rightarrow P$ of integral monoids is integral if the homomorphism of monoid algebras $\Z[Q]\rightarrow\Z[P]$ is flat.
The converse is true, if $h$ is local and $Q$ is sharp.
\end{lem}

Recall that a homomorphism $h\colon Q\rightarrow P$ of integral monoids said to be {\em exact} if $(h^{\gp})^{-1}(P)=Q$, where $h^{\gp}\colon Q^{\gp}\rightarrow P^{\gp}$ is the associated group homomorphism.
Recall also that, for a monoid $P$ and a submonoid $N\subset P$, the quotient monoid $P/N$ is given by $P/\sim$ (endowed with the natural monoid structure), where $a\sim b$ if and only if $ac=bd$ for some $c,d\in N$.

\begin{proof}
For (1), (4), and (6), see \cite[I.4.6.3 \& I.4.6.7]{AG2}.
(2) is immediate from the definition of integral homomorphisms.
(5) follows from (1), (3), and the fact that a homomorphism of integral monoids of the form $Q\rightarrow\ovl{Q}$ is always exact.
Hence it suffices to show (3).
To show that $\pi\colon P\rightarrow P/N$ ($a\mapsto\ovl{a}$)  is integral, take $a_1,a_2, b_1,b_2\in P$ such that $\ovl{a_1}\ovl{b_1}=\ovl{a_2}\ovl{b_2}$; the last equality means $a_1b_1c_1=a_2b_2c_2$ for $c_1,c_2\in N$, and if we set $a_3=b_1c_1$, $a_4=b_2c_2$ and $b=\ovl{1}$, then we have $a_1a_3=a_2a_4$, $\ovl{b_1}=\ovl{a_3}b$ and $\ovl{b_2}=\ovl{a_4}b$, which shows that $\pi$ is integral.
\end{proof}

\begin{dfn}\label{emdm2}
A morphism $f\colon X\rightarrow Y$ of integral log schemes is said to be {\em integral} {\em at $x\in X$} if the monoid homomorphism $\ovl{\M}_{Y,\ovl{y}}\rightarrow\ovl{\M}_{X,\ovl{x}}$, where $y=f(x)$, is integral, or equivalently, $\M_{Y,\ovl{y}}\rightarrow\M_{X,\ovl{x}}$ is integral (cf.\ Lemma \ref{lem-im} (6)).
We say $f$ is {\em integral} if it is integral at all points of $X$.
\end{dfn}

By Lemma \ref{lem-im} (2),  integral morphisms are stable under base change in the category of fine log schemes.
It is, however, not true that integral morphisms are stable under base change in the category of fs log schemes, cf.\ \cite[I.4.6.5]{AG2}.
So it is often convenient to refer to a base change (or a fiber product) in the category of fs log schemes as {\em fs base change} for emphasis.

Let us finally mention some technical facts on charts.
\begin{lem}\label{lem-goodcharts}
Let $f\colon X\rightarrow Y$ be a morphism of fs log schemes.
Then, for any $x\in X$ and $y=f(x)$, there exists commutative diagram 
$$
\xymatrix{X\ar[d]_f&U\ar[l]\ar[r]\ar[d]_g&\LSp(P)\ar[d]\\ Y&V\ar[l]\ar[r]&\LSp(Q)}
$$
comprised of an \'etale neighborhood $V\rightarrow Y$ of $\ovl{y}$, an fppf neighborhood $U\rightarrow X$ of $\ovl{x}$, and a chart $Q\rightarrow P$ of $g\colon U\rightarrow V$ such that the following conditions are satisfied:
\begin{itemize}
\item[(a)] $P$ and $Q$ are fs monoids;
\item[(b)] $Q\cong\ovl{\M}_{Y,\ovl{y}}$, and $\ovl{P}\cong\ovl{\M}_{X,\ovl{x}}$;
\item[(c)] $Q^{\gp}\rightarrow P^{\gp}$ is injective and $P^{\gp}/Q^{\gp}\cong\M^{\gp}_{X/Y,\ovl{x}}$.
\end{itemize}
$($Note that, in this situation, $Q$ is sharp, and $Q\rightarrow P$ is local.$)$
\end{lem}

\begin{proof}
Since $Y$ is fs, one can take on an \'etale neighborhood of $\ovl{y}$ a chart subordinate to the fs monoid $Q=\ovl{\M}_{Y,\ovl{y}}$.
Then one can construct a local chart of $f$ as above according to the recipe as in \cite[III.1.2.7]{AG2}, where one can take $P$ to be an fs monoid by the construction as in \cite[II.2.4.4]{AG2}
\end{proof}

\begin{lem}\label{lem-emim2}
Let $f\colon X\rightarrow Y$ be a morphism of fine log schemes.

$(1)$ Suppose $f$ has a $($global$)$ chart by a homomorphism $h\colon Q\rightarrow P$ of fine monoids, such that the conditions {\rm (b)} and {\rm (c)} in Lemma {\rm \ref{lem-goodcharts}} is satisfied.
Then $f$ is integral if and only if the homomorphism $h$ is integral.

$(2)$ If $f$ is integral at $x$, then it is integral at all points in an open neighborhood of $x$.
\end{lem}

\begin{proof}
See \cite[III.2.5.2]{AG2}.
\end{proof}

\section{Log blow-ups}\label{sec-lbu}
In this section, we briefly recall the notion of log blow-ups and their basic properties.

Recall first that an {\em ideal} of a monoid $M$ is a subset $K\subset M$ such that $x\in K$ and $a\in M$ imply $ax\in K$.
Trivial ideals are $\emptyset$ and $M$ itself. 
It follows from Dickson's lemma that any ideal of a finitely generated monoid is finitely generated.
If $\pi\colon M\rightarrow\ovl{M}$ is the canonical map, the map $K\mapsto\pi(K)$ gives a bijection from the set of ideals of $M$ to the set of ideals of $\ovl{M}$.

Let $X$ be a log scheme. 
A {\em log ideal} of $X$ is a sheaf of ideals $\mathcal{K}$ of $\M_X$.
We denote by $\ovl{\mathcal{K}}$ the corresponding ideal of $\ovl{\M}_X$.
For a morphism $f\colon X\rightarrow Y$ of fine log schemes and a log ideal $\mathcal{K}$ of $Y$, one has the extension of the log ideal $\mathcal{K}\M_X=(f^{-1}\mathcal{K})\M_X$, which is a log ideal of $X$.

\begin{exa}\label{exa-li}
Let $P$ be a monoid and $K\subset P$ an ideal.
One has the log ideal $\til{K}$ associated to $K$ on $X=\LSp(P)$, constructed as follows.
For any open subset $U\subset X$, we have a monoid homomorphism $P\rightarrow\M_X(U)$, and hence the extension ideal $K\M_X(U)$ of $\M_X(U)$.
Then $\til{K}$ is the sheafification of the subpresheaf of ideals of $\M_X$ given by $U\mapsto K\M_X(U)$.
Note that, for any $x\in X$, we have $\til{K}_{\ovl{x}}=K\M_{X,\ovl{x}}$.
\end{exa}

\begin{dfn}\label{dfn-cli}
A log ideal $\mathcal{K}$ of $X$ is called {\em coherent at $x\in X$} if there exists a local chart $U\rightarrow\LSp(P)$, where $U$ is an \'etale neighborhood around $\ovl{x}$, and an ideal $K\subset P$ such that $\mathcal{K}|_U=\til{K}\M_U$ (let us say, in this situation, that $K$ is a {\em chart} of $\mathcal{K}$ over $U$).
A log ideal is called {\em coherent} if it is coherent at all points.
\end{dfn}

\begin{rem}[{\rm cf.\ \cite[II.2.6.2]{AG2}}]\label{rem-cli}
If a log ideal $\mathcal{K}$ of $X$ is coherent at $x\in X$, then, for any local chart $U\rightarrow\LSp(P)$ around $\ovl{x}$, the pullback ideal $K\subset P$ of $\ovl{\mathcal{K}}_{\ovl{x}}$ by $P\rightarrow\M_{X,\ovl{x}}\rightarrow\ovl{\M}_{X,\ovl{x}}$ generates $\mathcal{K}$ around $\ovl{x}$; i.e., $\til{K}\M_{U'}=\mathcal{K}|_{U'}$ over an \'etale neighborhood $U'$ of $\ovl{x}$ contained in $U$.
\end{rem}

A coherent log ideal $\mathcal{K}$ of a log scheme $X$ is said to be {\em invertible} if, for any $x\in X$, $\ovl{\mathcal{K}}_{\ovl{x}}$ is a principal (i.e.\ generated by a single element) ideal of $\ovl{\M}_{X,\ovl{x}}$, or equivalently, there exist \'etale locally a chart $U\rightarrow\LSp(P)$ and an ideal $K\subset P$ as in Definition \ref{dfn-cli} with $K$ being principal. 

\begin{dfn}\label{dfn-lb}
A morphism $f\colon X'\rightarrow X$ of fine (resp.\ fs) log schemes is called a {\em log blow-up} along a coherent log ideal $\mathcal{K}$ if it has the following universal mapping property:
\begin{itemize}
\item[(a)] $\mathcal{K}\M_{X'}$ is an invertible log ideal of $\M_{X'}$;
\item[(b)] If $g\colon T\rightarrow X$ is a morphism of fine (resp.\ fs) log schemes such that $\mathcal{K}\M_T$ is invertible, then there exists a uniquely morphism $T\rightarrow X'$ such that the diagram 
$$
\xymatrix{T\ar[r]\ar[rd]_g&X'\ar[d]^f\\ &X}
$$
commutes.
\end{itemize}
\end{dfn}

The log blow-up along a coherent log ideal is unique up to isomorphisms.
Since every extension of invertible log ideal is again invertible, we have:
\begin{lem}\label{lem-lbbc}
The family of log blow-ups is stable under base change.
More precisely, if $X_{\mathcal{K}}\rightarrow X$ is a log blow-up of a fine log scheme $X$ along a coherent log ideal $\mathcal{K}$, and if $Y\rightarrow X$ is a morphism of fine log schemes, then $X_{\mathcal{K}}\times_XY\rightarrow Y$ is a log blow-up of $Y$ along $\mathcal{K}\M_Y$.
\end{lem}

If $X$ is an fs log scheme, and $f\colon X'\rightarrow X$ is a log blow-up in the category of fine log schemes, then the saturation $f^{\sat}\colon X^{\prime\sat}\rightarrow X$ gives a log blow-up of $X$ along the same coherent log ideal in the category of fs log schemes.
Hence, to show the existence of log blow-ups, it suffices to deal with the case of fine log schemes.

\begin{rem}\label{rem-normalization}
As indicated above, the log blow-ups in the category of fs log schemes are rather similar to {\em normalized blow-ups}, i.e., blow-up followed by normalization.
To make clear the distinction, we will call log blow-ups taken in the category of fs log schemes {\em fs log blow-ups}.
\end{rem}

The following construction of log blow-ups is due to Kazuya Kato \cite[(1.3.3)]{KK3} (cf.\ \cite[III.2.6]{AG2}):
We first construct the log blow-up 
$$
\BSp_K(P)\longrightarrow\LSp(P)
$$
of $\LSp(P)$, where $P$ is a fine monoid, along the coherent log ideal $\til{K}$ constructed from an ideal $K\subset P$.
Let $I(K)$ be the ideal of $\Z[P]$ generated by $K$, and consider the natural morphism $\Proj\bigoplus_nI(K)^n\rightarrow\Spec\Z[P]$.
$\Proj\bigoplus_nI(K)^n$ has the affine open covering
$$
\Proj\bigoplus_nI(K)^n=\bigcup_{t\in K}\Spec\Z[P\langle t^{-1}K\rangle].
$$
Here, $P\langle E\rangle$ for a subset $E$ of $P^{\gp}$ denotes the smallest fine submonoid in $P^{\gp}$ that contains $P$ and $E$.
The canonical log structures given by $P\langle t^{-1}K\rangle\rightarrow\Z[P\langle t^{-1}K\rangle]$ glue to a fine log structure on $\Proj\bigoplus_nI(K)^n$.
Then it follows that $\BSp_K(P):=\Proj\bigoplus_nI(K)^n\rightarrow\LSp(P)$ gives a log blow-up of $\LSp(P)$ along $\til{K}$.

To give a more explicit local description, take generators $t_0,\ldots,t_r$ of $K$.
Then $\BSp_K(P)$ is the union of the affine log schemes
$$
\Spec\Z\big[P\big\langle\textrm{{\footnotesize $\frac{t_0}{t_i},\ldots,\frac{t_r}{t_i}$}}\big\rangle\big],
$$
with the log structure induced from $P\langle t_0/t_i,\ldots,t_r/t_i\rangle\rightarrow\Z[P\langle t_0/t_i,\ldots,t_r/t_i\rangle]$, i.e., the affine log schemes $\LSp(P\langle t_0/t_i,\ldots,t_r/t_i\rangle)$, for $i=0,\ldots,r$.

Let $X$ be a fine log scheme, and $\mathcal{K}$ a coherent log ideal of $X$.
Suppose there exist a chart $\lambda\colon X\rightarrow\LSp(P)$ modeled on a fine monoid $P$ and an ideal $K\subset P$ such that $\mathcal{K}=\til{K}\M_X$.
Then, by Lemma \ref{lem-lbbc}, 
$$
\BSp_{\mathcal{K}}(X)=X\times_{\LSp(P)}\BSp_K(P)\longrightarrow X
$$
gives a log blow-up of $X$ along $\mathcal{K}$.

In general, we take an \'etale covering $\{U_{\alpha}\}_{\alpha\in L}$ of $X$ such that each $U_{\alpha}$ allow a chart $U_{\alpha}\rightarrow\LSp(P_{\alpha})$ with an ideal $K_{\alpha}\subset P_{\alpha}$ satisfying $\mathcal{K}|_{U_{\alpha}}=\til{K}_{\alpha}\M_{U_{\alpha}}$.
Then, by the universality of log blow-ups, the local log blow-ups $\BSp_{\mathcal{K}_{\alpha}}(U_{\alpha})\rightarrow U_{\alpha}$ constructed as above glue to a log blow-up of $X$ along $\mathcal{K}$.

\begin{exa}\label{exa-tlb}
Let $P$ be a sharp fs monoid, and set $X=\LSp(P)$.
Set $M=P^{\gp}$ and $N=\Hom_{\Z}(M,\Z)$.
The scheme $X$ is an affine toric variety corresponding to the corn $\sigma$ in $N_{\R}$ such that $\sigma^{\vee}\cap M=P$; i.e., $\sigma$ is the dual corn of the corn in $M_{\R}$ generated by $P$.
Let $\phi\colon\sigma\rightarrow\R_{\geq 0}$ be a continuous convex piecewise linear function satisfying the following conditions (cf.\ \cite[p.27]{KKM}):
\begin{itemize}
\item[(a)] $\phi(\lambda x)=\lambda\phi(x)$ for $x\in\sigma$ and $\lambda\in\R_{\geq 0}$;
\item[(b)] $\phi(N\cap\sigma)\subset\Z$.
\end{itemize}
The function $\phi$ induces an ideal $K_{\phi}$ of $P$ given by
$$
K_{\phi}=\{m\in M\mid\langle x,m\rangle\geq\phi(x)\ \textrm{for all}\ x\in\sigma\}.
$$
Then the fs log blow-up of $X$ along the log ideal $\til{K}_{\phi}$ is the normalization of the blow-up of the toric variety $X=X_{\sigma}$ obtained from the coarsest subdividing fan $\Sigma_{\phi}$ of the cone $\sigma$  such that $\phi$ is linear on each cone in $\Sigma_{\phi}$; cf.\ \cite[p.31,\ Theorem 10]{KKM}.
\end{exa}

\section{Proof of the theorem}\label{sec-proof}
\begin{lem}\label{lem-mt}
Let $P,Q$ be sharp fs monoids, and $h\colon Q\hookrightarrow P$ an injective homomorphism.
Consider the induced morphism $f\colon X=\LSp(P)\rightarrow Y=\LSp(Q)$ of fs log schemes.
Then there exists an ideal $K\subset Q$ such that the following conditions are satisfied: if
$$
\xymatrix{X'\ar[d]_{f'}\ar[r]&X\ar[d]^f\\ \BSp_K(Y)\ar[r]&Y}
$$
is the fs base change of $f$ by the fs log blow-up $\BSp_K(Y)\rightarrow Y$, then the underlying scheme-theoretic morphism of $f'$ is equidimensional.
\end{lem}

Note that $X'\rightarrow X$ is isomorphic to the fs log blow-up along $\til{KP}=\til{K}\M_X$, i.e., $X'\cong\BSp_{KP}(X)$.

\begin{proof}
In this proof, we follow the original argument in \cite[3.16]{FK} based on the idea of T.\ Kajiwara, which we note is similar to the argument in \cite[Lemma 4.3]{AK}.

We use the following notation:
\begin{itemize}
\item $M_Q=Q^{\gp}$, $M_P=P^{\gp}$;
\item $N_Q=\Hom_{\Z}(M_Q,\Z)$, $N_P=\Hom_{\Z}(M_P,\Z)$;
\item $\sigma_Q$ $($resp.\ $\sigma_P)$  $=$ the cone in $N_Q$ $($resp.\ $N_P)$ such that $Q=\sigma^{\vee}_Q\cap M_Q$ $($resp.\ $P=\sigma^{\vee}_P\cap M_P)$.
\item $\Sigma_Q$ (resp. $\Sigma_P$) $=$ the fan made up from the faces of the cone $\sigma_Q$ (resp.\ $\sigma_P$).
\end{itemize}
Note that we have a map $\Phi\colon\Sigma_P\rightarrow\Sigma_Q$ of fans that induces the morphism of affine toric schemes $\Spec\Z[P]\rightarrow\Spec\Z[Q]$.

Consider the subset $\Sigma_{P,1}\subset\Sigma_P$ (resp.\ $\Sigma_{Q,1}\subset\Sigma_Q$) of rays, i.e., cones of dimension $1$.
Each $\rho\in\Sigma_{P,1}$ is mapped by $\Phi$ onto a ray in $\Sigma_Q$ or to a point (i.e.\ the origin of $N_Q$).
If $\rho$ is mapped onto a ray, then take the primitive base $n_1\in N_Q$ of $\Phi(\rho)$, and extend it to a $\Z$-base $n_1,\ldots,n_r$ of $N_Q$.
The $r+1$-rays spanned by $n_1,\ldots,n_r,-(n_1+\cdots+n_r)$ defines the projective $r$-space $\P^r_{\Z}$, and hence the ideal $\OO(-1)$ gives rise to a support function, denoted by $\phi_{\rho}$, i.e., a continuous convex piecewise linear function $N_{Q,\R}\rightarrow\R_{\geq 0}$ satisfying the conditions (a) and (b) in Example \ref{exa-tlb}; we denote the restriction of $\phi_{\rho}$ onto $\sigma_Q$ by the same symbol.
If $\Phi(\rho)$ is a point, then set $\phi_{\rho}=0$.
Set
$$
\phi=\sum_{\rho\in\Sigma_{P,1}}\phi_{\rho},
$$
and let $\Sigma'_Q$ be the coarsest fan that subdivides $\Sigma_Q$ such that $\phi$ is linear on each cone in $\Sigma'_Q$. 
(The author learned this way of constructing $\phi$ from T.\ Kajiwara.)

Now, let $K\subset Q$ be an ideal constructed from $\phi$ as in Example \ref{exa-tlb}.
Consider the fs log blow-up $\BSp_K(Y)\rightarrow Y$, and the fs base change $f'\colon X':=X\times_Y\BSp_K(Y)\rightarrow\BSp_K(Y)$ of $f$.
Then, $X'\rightarrow X$ is the log blow-up of $X$ along $\til{K}\M_X$, which is the normalized toric blow-up induced from the piecewise linear function on $N_{P,\R}$ given by the pull-back of $\phi$.
If we denote the corresponding fan of $X\times_Y\BSp_K(Y)$ by $\Sigma'_P$, then the induced map $\Phi'\colon\Sigma'_P\rightarrow\Sigma'_Q$ maps each ray onto either a ray or a point (the origin), and hence mapping each cone onto a cone.
Therefore, the morphism $f'\colon X'\rightarrow\BSp_K(Y)$ is equidimensional by \cite[Lemma 4.1]{AK}.
\end{proof}

\begin{proof}[Proof of Theorem {\rm \ref{thm-main}}]
Let $Q=\ovl{\M}_{Y,\ovl{y}}$, and take an \'etale local chart $Y\leftarrow V\rightarrow\LSp(Q)$ around $\ovl{y}$.
For any $x\in X_V=X\times_YV$, take an fppf local chart $X_V\leftarrow U_x\rightarrow\LSp(P_x)$ around $\ovl{x}$, where $P_x$ is an fs monoid, which extends to a local chart of $f$ as in Lemma \ref{lem-goodcharts}.
Since $f$ is quasi-compact, one can take finitely many $x_1,\ldots,x_n\in X_V$ such that $X_V$ is covered by the union of $U_i:=U_{x_i}$ for $i=1,\ldots,n$.
We set $P_i=P_{x_i}$ for $i=1,\ldots,n$.

For $i=1,\ldots,n$, there exists by Lemma \ref{lem-mt} an ideal $K_i\subset Q$ such that the fs base change $\BSp_{K_iP_i}(P_i)=\LSp(P_i)\times_{\LSp(Q)}\BSp_{K_i}(Q)\rightarrow\BSp_{K_i}(Q)$ by the corresponding log blow-up is equidimensional.
Set $K=K_1\cdots K_n$.
Then $\BSp_{KP_i}(P_i)=\LSp(P_i)\times_{\LSp(Q)}\BSp_K(Q)\rightarrow\BSp_K(Q)$ is equidimensional for any $i=1,\ldots,n$.

One can further perform a toric blow-up of the toric scheme $\BSp_K(Q)$ so that the resulting toric scheme is smooth over $\Z$ (cf.\  \cite[p.32,\ Theorem 11]{KKM}).
Since the composition of fs log blow-ups is again an fs log blow-up (\cite[4.11]{Niz}), we may assume that there exists an ideal $K\subset Q$ such that
\begin{itemize}
\item[(a)] the induced morphism 
$$
\BSp_{KP_i}(P_i)=\LSp(P_i)\times_{\LSp(Q)}\BSp_K(Q)\longrightarrow\BSp_K(Q)\eqno{(\ast)_i}
$$
is equidimensional for each $i=1,\ldots,n$; 
\item[(b)] $\BSp_K(Q)$ is smooth over $\Z$.
\end{itemize}

We claim that $(\ast)_i$ is integral for  each $i=1,\ldots,n$.
Since toric schemes are Cohen-Macaulay, the properties (a) and (b) imply that the underlying scheme-theoretic morphism of $(\ast)_i$ is flat.
Thus, for any cones $\sigma$ from the fan of $\BSp_K(Q)$ and $\tau$ from the fan of $\BSp_{KP_i}(P_i)$ such that $\tau$ is mapped to $\sigma$, the affine portion of $(\ast)_i$
$$
\LSp(P'_i)\longrightarrow\LSp(Q')
$$
where $Q'=\sigma'\cap M_Q$ and $P'_i=\tau'\cap M_P$, is flat, and hence is integral by Lemma \ref{lem-im} (6) and Lemma \ref{lem-emim2} (2).
This means $(\ast)_i$ is integral.

Now, by Lemma \ref{lem-emim2} (2), we deduce that $U_i\rightarrow V$ is integral for any $i=1,\ldots,n$, and hence $X_V\rightarrow V$ is integral.
\end{proof}

Let us finally remark that, the argument of the above proofs shows that, if we start from a {\em toroidal morphism} (in the sense as in \cite[\S1]{AK}) $f\colon X\rightarrow Y$, then, since $f$ is described globally by a morphism of polyhedral complexes $f_{\Delta}\colon\Delta_X\rightarrow\Delta_Y$ of K.\ Kato's fans (cf.\ \cite{KK4}), one can actually do the above argument globally on $Y$; cf.\ \cite[4.4]{AK}.
Since the only question here lies as to whether one can take a global log blow-up of $Y$, one can slightly generalize the situation to $Y$ being log regular but without assuming $f$ to be toroidal.
This situation has been treated in \cite{ATW}, which we include here for the reader's convenience:
\begin{prop}[{\cite[3.6.11]{ATW}}]\label{prop-ATW}
Let $f\colon X\rightarrow Y$ be a quasi-compact morphism of fs log schemes, where $Y$ is log regular.
Then there exists an fs log blow-up $Y':=\BSp_{\mathcal{K}}(Y)\rightarrow Y$ along a coherent log ideal $\mathcal{K}\subset\M_Y$ such that the fs base change $f_{Y'}\colon X_{Y'}\rightarrow Y'$ is integral.
\end{prop}

\frenchspacing
\begin{small}

\medskip\noindent
{\sc Department of Mathematics, Tokyo Institute of Technology, 2-12-1 Ookayama, Meguro, Tokyo 152-8551, Japan} (e-mail: {\tt bungen@math.titech.ac.jp})
\end{small}

\begin{thebibliography}{99}
\bibitem{AK}Abramovich, D.; Karu, K.: {\it Weak semistable reduction in characteristic $0$}, Invent.\ Math.\ {\bf 139} (2000), no.\ 2, 241--273.
\bibitem{ATW}Abramovich, D.; Temkin, M.; Wlodarczyk, J.: {\it Relative Desingularization and principalization of ideals}, preprint, arXiv:2003.03659.
\bibitem{Gil}Gillam, W.D.: {\it Logarithmic flatness}, preprint, arXiv:1601.02422.
\bibitem{IKN} Illusie, L.; Kato, K.; Nakayama, C.: {\it Quasi-unipotent Logarithmic Riemann-Hilbert Correspondences}, J.\ Math.\ Sci.\ Univ.\ Tokyo {\bf 12} (2005), 1--66.
\bibitem{FK} Kato, F.: {\it Exactness, integrality, and log modifications}, preprint, arXiv:math/9907124.
\bibitem{KK1} Kato, K.: {\it Logarithmic structures of Fontaine--Illusie}, in Algebraic Analysis, Geometry and Number Theory (J.-I.\ Igusa, ed.). Johns Hopkins Univ., 1988, 191--224. 
\bibitem{KK3} Kato, K.: {\it Logarithmic degeneration and Dieudonne theory}, preprint.
\bibitem{KK4} Kato, K.: {\it Toric singularities}, Amer.\ J.\ Math.\ {\bf 116} (1994), no.\ 5, 1073--1099. 
\bibitem{KKM} Kempf, G., Knudsen, F., Mumford, D., Saint-Donat, B.: {\it Toroidal Embeddings I}, Lecture Note in Math.\ {\bf 339}, Springer-Verlag, Berlin, Heidelberg, New York (1973).
\bibitem{Niz}Nizio\l, W.: {\it Toric singularities: log-blow-ups and global resolutions}, Journal of Algebraic Geometry {\bf 15} (2006), 1--29.
\bibitem{AG2} Ogus, A.: {\it Lectures on Logarithmic Algebraic Geometry}, Cambridge University Press, Nov.\ 2018.
\bibitem{Ols}Olsson, M.C.: {\it Logarithmic geometry and algebraic stacks}, Ann.\ Sci.\ \'Ecole Norm.\ Sup.\ (4) {\bf 36} (2003), no.\ 5, 747--791.
\bibitem{TT} Tsuji, T.: {\it Saturated morphisms of logarithmic schemes}, Tunis.\ J.\ Math.\ {\bf 1} (2019), no.\ 2, 185--220.
\end{thebibliography}
\end{document}